\documentclass[11pt]{amsart}
\catcode`\_=13
\def_#1{\sb{\sb{#1}}}

\usepackage{t1enc}
\usepackage{amsthm,amssymb}
\usepackage{latexsym}
\usepackage{amsmath} 
\usepackage{graphics}
\usepackage{epsfig,multicol}
\usepackage{amsfonts}
\usepackage[latin1]{inputenc}
\usepackage[english]{babel}
\usepackage{ mathrsfs }

\newtheorem{theorem}{Theorem}[section]
\newtheorem{proposition}{Proposition}[section]
\newtheorem{lemma}{Lemma}[section]

\newtheorem{definition}{Definition}[section]

\newtheorem{myfigure}{Figure}[section]
\newtheorem{myequation}{Equation}[section]

{\theoremstyle{definition} \newtheorem{example}{Example}[section]}

\newtheorem{formula}{Formula}[section]

\def\hpic #1 #2 {\mbox{$\begin{array}[c]{l} \epsfig{file=#1,height=#2}
\end{array}$}}
 
\def\vpic #1 #2 {\mbox{$\begin{array}[c]{l} \epsfig{file=#1,width=#2}
\end{array}$}}

\newcommand  {\rmn}\romannumeral

\newcommand {\5}{\vskip 5pt}
\begin{document}
\title[On spectral measures for certain unitary representations of $F$]{On spectral measures for certain unitary representations of R. Thompson's group F.}
\author{Valeriano Aiello and Vaughan F. R. Jones}
\address{Valeriano Aiello\\ 
Section de Math\'  ematiques, Universit\' e de Gen\` eve, 2-4 rue du Li\` evre, Case Postale 64,
1211 Gen\` eve 4, Switzerland
}\email{valerianoaiello@gmail.com}
\address{Vaughan F. R. Jones\\ Vanderbilt University, Department of Mathematics, 1326 Stevenson Center Nashville, TN, 37240, USA}\email{vaughan.f.jones@vanderbilt.edu}
\thanks{
V.A. is supported by the  Swiss National Science Foundation.
V.J. is supported by the grant numbered DP140100732, Symmetries of subfactors.}

\begin{abstract}
The Hilbert space $\mathcal H$ of backward renormalisation  of an anyonic quantum spin chain affords a 
unitary representation of Thompson's group $F$ via local scale transformations. Given a vector in the canonical dense
subspace of $\mathcal H$ we show how to calculate the corresponding spectral measure for any element of $F$ and
illustrate with some examples. Introducing the "essential part" of an element we show that the spectral measure of any
vector in $\mathcal H$ is, apart from possibly finitely many  eigenvalues, absolutely continuous with respect to Lebesgue
measure. The same considerations and results hold for the Brown-Thompson groups $F_n$ (for which $F=F_2$).

\end{abstract}
\maketitle
\tableofcontents

\section{Introduction.}
Let $F$ and $T$ be the Thompson groups as usual. In \cite{Jo16} an action of $F$ was shown to arise from a \emph{functor} from the category $\mathcal F$  whose objects are natural numbers and whose morphisms are 
planar binary forests, to another category $\mathcal C$.
Forests decorated with cyclic permutations of their leaves give a category $\mathcal T$ for which functors from $\mathcal T$ give actions of $T$.

The representations studied in \cite{Jo16} came from functors $\Phi$ to a trivalent tensor category (planar algebra) $\mathcal C$ in
the sense of \cite{MPS}, based on a specific "vacuum vector" $\Omega$ in the 1-box space of the tensor category.  A Thompson group 
element $g$ is represented by a pair of binary planar trees, drawn in the plane with one tree upside down on top of the other as below for an
element of $F$ that we will call X:
\begin{eqnarray}\label{elementX}
X= \vpic {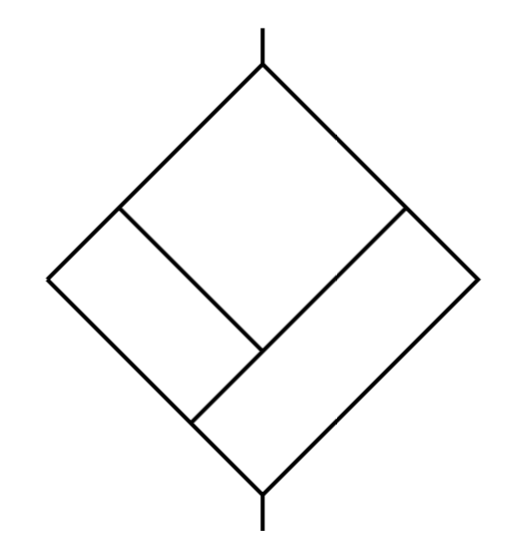} {1.5in}
\end{eqnarray}
 
The standard dyadic intervals defined by the leaves of the bottom tree are sent by $g$ (in the only affine way possible) to the corresponding intervals for the top tree. 

 If $\pi$ is the unitary representation defined by the (suitably normalised) trivalent vertex in $\mathcal C$, the coefficient 
 $$\langle \pi(g)\Omega, \Omega \rangle$$ is simply equal to the pair of trees of $g$ interpreted as a planar diagram (tangle) for
 $\mathcal C$! (Or more correctly the pair of trees as drawn is a multiple of a single vertical straight line, and that multiple is
$\langle \pi(g)\Omega, \Omega \rangle$.)
\section{Definitions.}

 An $n$-ary planar forest is the isotopy class of a disjoint union of   
trees,   
whose vertices are $n+1$-valent, 
embedded
in $\mathbb R^2$, all of whose roots lie on $(\mathbb R,0)$ and all of whose leaves lie on $(\mathbb R,1)$. 
 The isotopies are supported
in the strip $(\mathbb R,[0,1])$. 
The $n$-ary planar forests form a category
in the obvious way with objects being $\mathbb N$ whose elements are identified with isotopy classes of sets of points on a line and whose morphisms are the forests which can be composed by stacking a forest in 
$(\mathbb R,[0,1])$ on top of another, lining up the leaves of the one on the bottom with 
the roots of the other by isotopy then rescaling the $y$  axis to return
to a forest in  $(\mathbb R,[0,1])$. The structure is of course actually combinatorial but it is very useful to think of it in the way we have described.

We will call this category $\mathcal F_n$ and from now on we will assume $n\geq 2$.

\begin{definition}\label{genforest} 
Fix $m\in \mathbb N$. For each $i=1,2,\cdots,m$ let $f_i$ be the planar
$n$-ary forest with $m$ roots and $m+n-1$ leaves consisting of straight
lines joining $(k,0)$ to $(k,1)$ for $1\leq k\leq i-1$ and $(k,0)$ to 
$(k+n-1,1)$ for $i+1\leq k\leq m$, and a single $n$-ary tree with root
$(i,0)$, leaves $(i,1)$, $(i+1,1)$, \ldots , $(i+n-1,1)$,   
thus:
 m at the bottom right and m+n-1 at top right
$$
\includegraphics[scale=0.5]{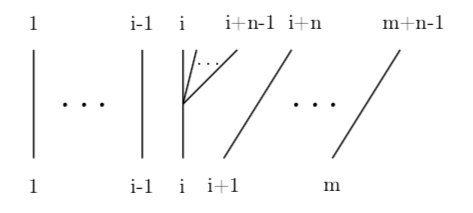}
$$
\end{definition}

Note that any element of $\mathcal F_n$ is in an essentially unique way a 
composition of morphisms $f_i$, the only relation being  
$f_j f_i= f_i f_{j-n+1}\mbox{   for  } i<j-n+1$.
The set of 
morphisms from $1$ to $k$ in $\mathcal F_n$ is the set  of $n$-ary planar rooted trees $\mathfrak T_n$ and is a \emph{directed set} with $s\leq t$ iff there is an $f\in \mathcal F_n$ with $t=fs$.

Given a functor $\Phi:\mathcal  F_n\rightarrow \mathcal C$ to a category $\mathcal C$ whose objects are sets, 
we define the direct system $S_\Phi$ which associates to each $t  \in \mathfrak T_n$, $t:1\rightarrow k$,
 the set $S_t:=\Phi(target(t))=\Phi(k)$. 
 For
each $s\leq t$ we need to give $\iota_s^t$. For this observe 
that there is an $f\in \mathcal F_n$ for which $t=fs$ so we
define, for $\kappa \in \Phi(target(s))$,
 $$\iota_s^t =\Phi(f)$$
 which is an element of $Mor_{\mathcal C}(\Phi(target(s)),\Phi(target(t)))$ as required. The $\iota_s^t$ trivially
 satisfy the axioms of a direct system.
 
  As a slight variation on this theme, given a functor $\Phi:\mathcal  F_n\rightarrow \mathcal C$ to \underline{any} category $\mathcal C$, 
and an object $\omega \in \mathcal C$, form the category $\mathcal C^\omega$ whose objects are the sets $Mor_{\mathcal C}(\omega,obj)$ for every 
object $obj$ in $\mathcal C$, and
whose morphisms are composition with those of $\mathcal C$. The definition of the functor $\Phi^\omega:\mathcal F_n\rightarrow \mathcal C^\omega$
is obvious.
  Thus the direct system $S_{\Phi^\omega}$ associates to each $t  \in \mathfrak T_n$, $t:1\rightarrow k$,
 the set $Mor_{\mathcal C}(\omega,\Phi(k))$. Given $s\leq t$ let  $f\in \mathcal F_n$ be such that $t=fs$. Then
for $\kappa \in Mor_{\mathcal C}(\omega,\Phi(target(s)))$,
 $$\iota_s^t(\kappa) =\Phi(f)\circ \kappa$$
 which is an element of $Mor_{\mathcal C}(\omega,\Phi(target(t)))$. 
 
 As in \cite{Jo16} we consider the direct limit:
$$ \underset{\rightarrow} \lim S_\Phi=\{(t, x) \mbox{ with } t \in  \mathfrak T_n, x\in \Phi(target(t))\} / \sim$$ 
 where $(t,x)\sim (s,y)$ iff there are $p, q\in \mathcal F_n$ with $pt=qs$ and $\Phi(p)(x)=\Phi(q)(y)$.
\vskip 5pt
\emph{We use $\displaystyle {t \over x}$ to denote the equivalence class of $(t,x)$ mod $\sim$.}
\vskip 5pt
The limit $ \underset{\rightarrow} \lim S_\Phi$ will inherit structure from
the category $\mathcal C$. For instance if the objects of $\mathcal C$ are Hilbert spaces and the morphisms are isometries then 
$ \underset{\rightarrow} \lim S_\Phi$ will be a pre-Hilbert space which may be completed to a Hilbert space which we will also call the direct
limit unless special care is required. We will denote this Hilbert space by $\mathcal H$.

As was observed in \cite{Jo16}, if we let 
 $\Phi$ be the identity functor and choose $\omega=1$,  
then  the inductive limit consists of equivalence classes of pairs $\frac{t}{x}$ where $t\in \mathfrak T_n$ and $x\in \Phi(target(t))=Mor(1,target(t))$.
But $Mor(1,target(t))$ is nothing but $s\in \mathfrak T_n$ with $target(s)=target(t)$, i.e. trees with the same number of leaves as $t$.
Thus the inductive limit is nothing but the Brown-Thompson group $F_n$ with group law $${r\over s}{s\over t}={r\over t}$$ (see \cite{Brown} for more information on the properties of these groups).   

 Moreover for any other functor $\Phi$,
$ \underset{\rightarrow} \lim S_\Phi$ carries a natural action of $F_n$ 
defined as follows:
$$\frac{s}{t}\left(\frac{t}{x}\right)=\frac{s}{x}$$ where $s,t\in \mathfrak T_n$ with $target(s)=target(t)=k$ and $x\in \Phi(k)$.  
A Brown-Thompson group element given as a pair of trees with $h$ leaves, and an element of
 $ \underset{\rightarrow} \lim S_\Phi$ given as a pair (tree with $k$ leaves, element of $\Phi(k)$), may not be immediately composable by the above formula, but they can always be ``stabilised'' to be so within their equivalence classes. 
 
 The Brown-Thompson group action preserves the structure of 
 $ \underset{\rightarrow} \lim S_\Phi$ so for instance in the Hilbert space case the actions define unitary representations.

We end this section by introducing an alternative description of the  
Brown-Thompson groups $F_n$, 
which is a straightforward extension of the 
one in
 \cite[Chapter 7]{B} and \cite{BM}
 where   Thompson's groups $F$, $T$, and $V$ are considered. The $F_n$ thus become diagram groups in the
 sense of \cite{GS}.
 \\ 
Let $h, k\in\mathbb{N}$, an $(h, k)$-strand diagram 
is a finite acyclic directed graph embedded in the
unit square $[0,1]\times [0,1]$, with the following properties
\begin{enumerate}
\item the graph has $h$ univalent sources along the top of the square, and $k$ univalent sinks along the bottom of the square;
\item every other vertex is 
$n+1$-valent, and is either a split or a merge (see Figure \ref{AAA}). 
\end{enumerate}

\begin{myfigure}
\label{AAA}
A merge and a split: 
 \vpic {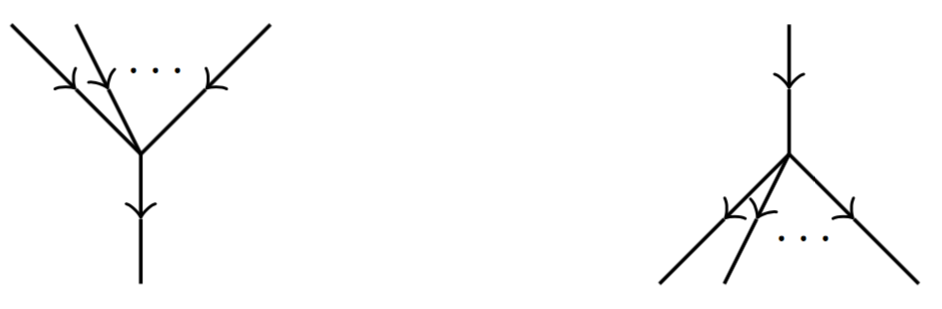} {3in}
\end{myfigure}
 
A reduction of an $(h,k)$-strand diagram is a move of type I and II shown in Figure \ref{reduction-fig}.

\begin{myfigure}\label{reduction-fig}
The three reduction moves. 

\qquad\vpic {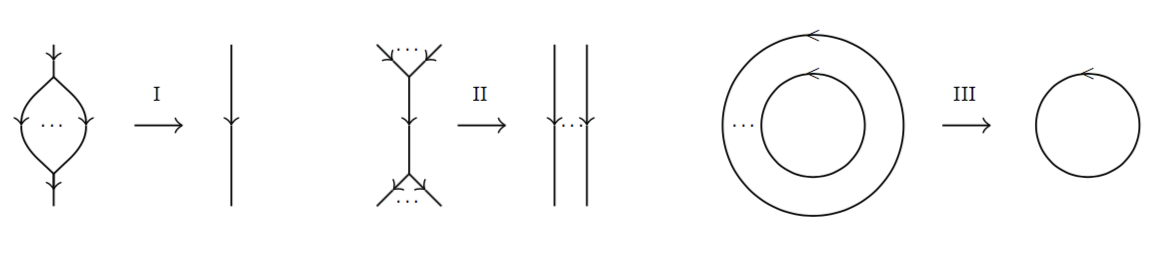} {4.5in}
\end{myfigure}

 An $(h,k)$-strand diagram is reduced if no more reductions are possible. 
Two $(h,k)$-strand diagrams are equivalent if one can be obtained from the other by a sequence of reductions and inverse reductions. 
We observe that every reduced $(1,1)$-strand diagram is uniquely determined by the composition of a tree and an inverse tree (see \cite[Theorem 7.1.6]{B} for the case $F=F_2$).  
Thanks to the strand diagrams, we have the following description of the Brown-Thompson groups $F_n$, cf. \cite[Proposition 2.5]{BM} and \cite{GS}.
\begin{theorem}  \label{BM-description-F}
Reduced 
$(1,1)$-strand diagrams form a groupoid over the positive integers, 
the composition 
being
 concatenation followed by reduction,
and the  inverse being 
reflection about a horizontal line. 
Moreover, the isotropy group 
at the point $1$ is isomorphic to the Brown-Thompson group $F_n$.
\end{theorem}

Since the orientation of the edges is usually clear from the context, the arrows will be often suppressed.

As Belk and Matucci did for $F=F_2$ in \cite[Section 2.2]{BM}, 
given a $(k,k)$-strand diagram $g$, we can \emph{close} it, that is we may join the sinks to the sources, and obtain a graph $g'$ which embeds into the annulus (see Figure \ref{closure}).
This graph is called the \emph{closure} of $g$.

\begin{myfigure}
\label{closure}
The closure of a strand diagram.
$g=$\vpic {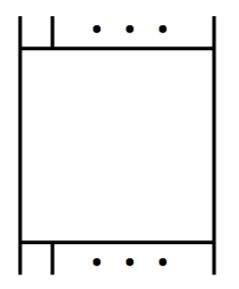} {.5in} 
$\mapsto g'=$\vpic {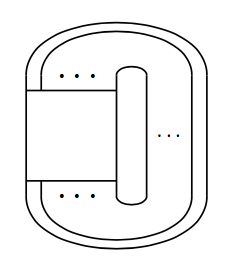} {1.1in} 
\end{myfigure}
The closure of a strand diagrams is, in particular, an \emph{annular strand diagram}. 
An annular strand diagram is 
a graph embedded in the annulus $[0,1]\times S^1$, without bivalent vertices, such that every vertex is a split or a merge and every directed cycle winds counterclockwise around the central hole \cite[Definition 2.9, p. 244]{BM}. \\
Given an annular strand diagram
we may apply the reduction moves mentioned above, along with a move of type III (see Figure \ref{reduction-fig}).
We stress that these moves should be interpreted in the relevant groups. This means that in the second move we have $n$ parallel  edges and in the third   we have $n$ concentric circles.
We mention that the operation of closure of strand diagrams was introduced by Belk and Matucci to provide a solution to the conjugacy problem in Thompson's group $F$. 
We would like to thank Dylan Thurston for pointing out to us the work of Belk and Matucci in \cite{BM}.

 \section{Planar algebras.}
 A planar algebra is a collection of vector spaces ($P_n$) whose elements can be combined in a multilinear way for every
 planar tangle - a planar tangle is a collection of disjoint ``input'' 
  discs inside an ``output''  disc. The discs are connected
 by smooth non-intersecting curves known as strings. Elements of $P_n$ ``go into'' an input disc with exactly $n$
 boundary points meeting the strings of the tangle.  Once all the input discs have vectors assigned to them, the output vector is
 in $P_m$ where $m$ is the number of points on the output disc meeting the strings of the tangle. Planar tangles can be glued one into another to produce new planar tangles and the
 operations on vectors are compatible with gluing. A planar algebra with dim($P_0)=1$ has a ``loop'' 
  parameter usually denoted  $d$ or $\delta$ such that any
 closed string (not meeting a disc of the tangle) can always be removed provided a multiplicative factor of $\delta$ is
 applied to the multilinear map. See \cite{jo2} for details.

For instance tensors  give a planar algebra where $P_n=\otimes^n \mathbb C^{k}$  (the indices of the tensors go from $1$ to $k$)
and the planar tangle is interpreted as a scheme for contracting the indices of the tensors, with indices assigned to strings. The loop
parameter here is $k$.

 If we take an $n$-ary planar forest $f$ with $p$ roots and $q$ leaves and surround it by an annulus with roots attached to the inner circle  and leaves to  the outer one, and enlarge to discs all the vertices of
 the forest we have a planar tangle all but one of whose input discs have $n+1$ strings attached to their boundaries. (The edges of the forest
 are the strings of the planar tangle.) We then choose a vector $R$ in the vector space $P_{n+1}$ and insert it at all the vertices of the forest, obtaining
 a labelled tangle which produces a linear map $\Phi_R(f)$ from $P_p$ to $P_q$ as below.
 \vskip 5pt
\hspace {40pt} \qquad \qquad Forest $f:3\rightarrow 6$ \qquad \qquad\quad $\Phi_R(f):P_3\rightarrow P_6$
 \vskip 5pt
\hspace {40pt} \qquad \qquad \vpic {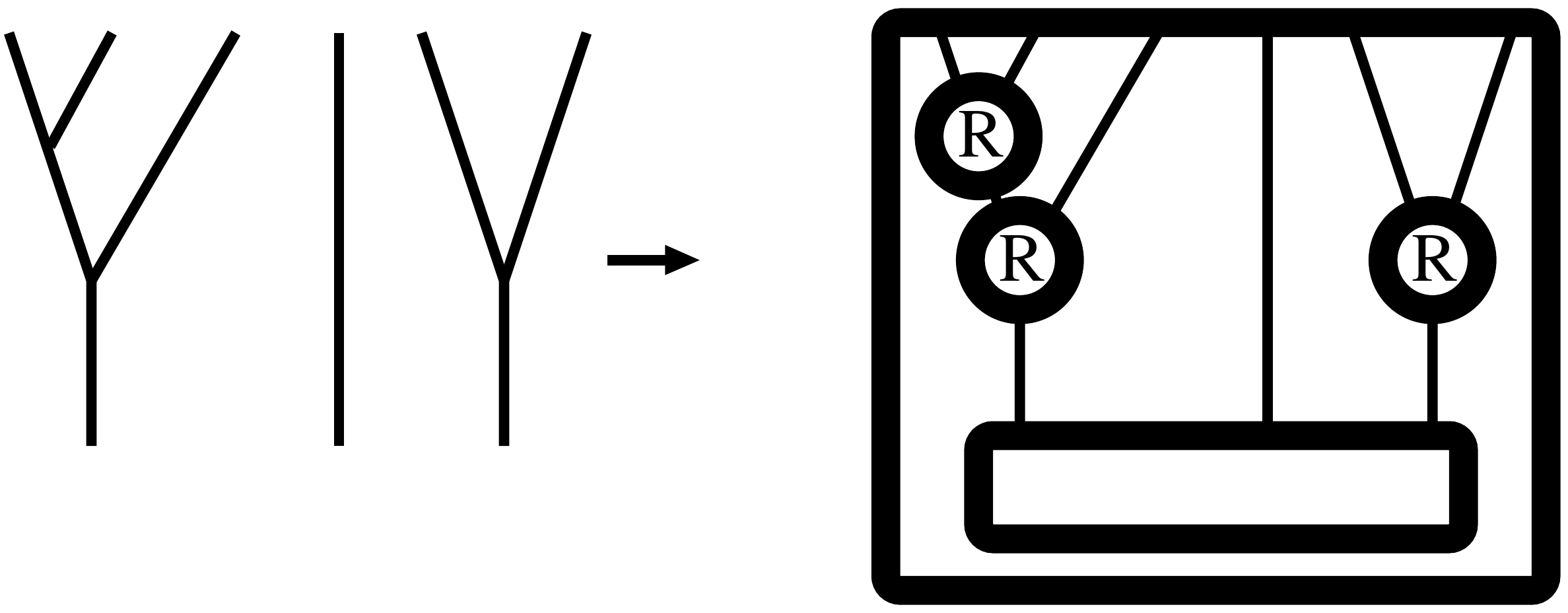} {3in}
 \vskip 5pt
 The axioms of a planar algebra are more than enough to show that the linear maps $\Phi_R(f): P_p\rightarrow P_q$ given by these
 tangles define a functor from $n$-ary forests to vector spaces and thus a linear representation of the Brown-Thompson group $F_n$.
 There is in general no simpler description of the vector space on which the representation acts than as the direct limit of the
 $P_n$ according to the directed set of rooted planar $n$-ary trees. 
 
  Nothing at all is required of the element $R$ above. To make the  linear representations unitary one requires more structure of the planar algebra. First each $P_n$ is endowed
 with an antilinear involution $*$ which is compatible with orientation reversing diffeomorphisms of tangles in the obvious way.
 We further require that the space $P_0$ is one dimensional so there is a sesquilinear form $\langle\; ,\; \rangle$ on each $P_n$
 defined by 
    $\mbox{                       }    \; \; \qquad\qquad \qquad\qquad\qquad  \quad \langle X, Y \rangle=$\vpic {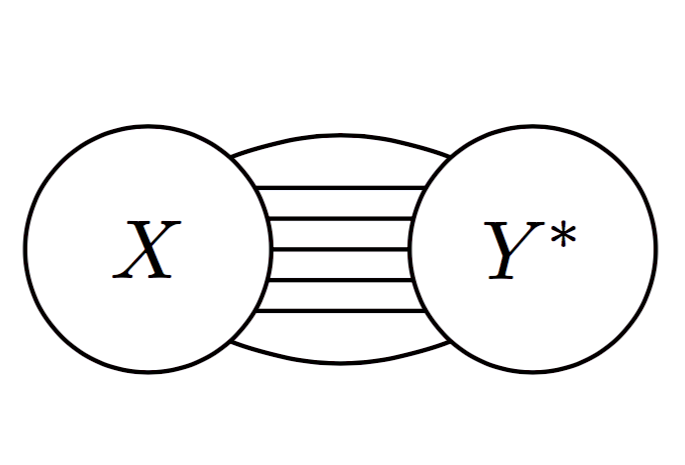} {1.35in} 
 \\
 where we have suppressed the output disc of the tangle since it does not meet any strings.
 A planar algebra is called a \emph{Hilbert planar algebra} if the resulting inner product is positive definite.
 
 If $(P_n)$ is a Hilbert planar algebra and we choose an element $R\in P_{n+1}$ it is clear that the representation of the
 Brown-Thompson group $F_n$ which we have constructed using $R$ will be unitary if
 
\begin{myequation}\label{unitarity}
 
 \hspace {6pt} \qquad\quad\qquad \vpic {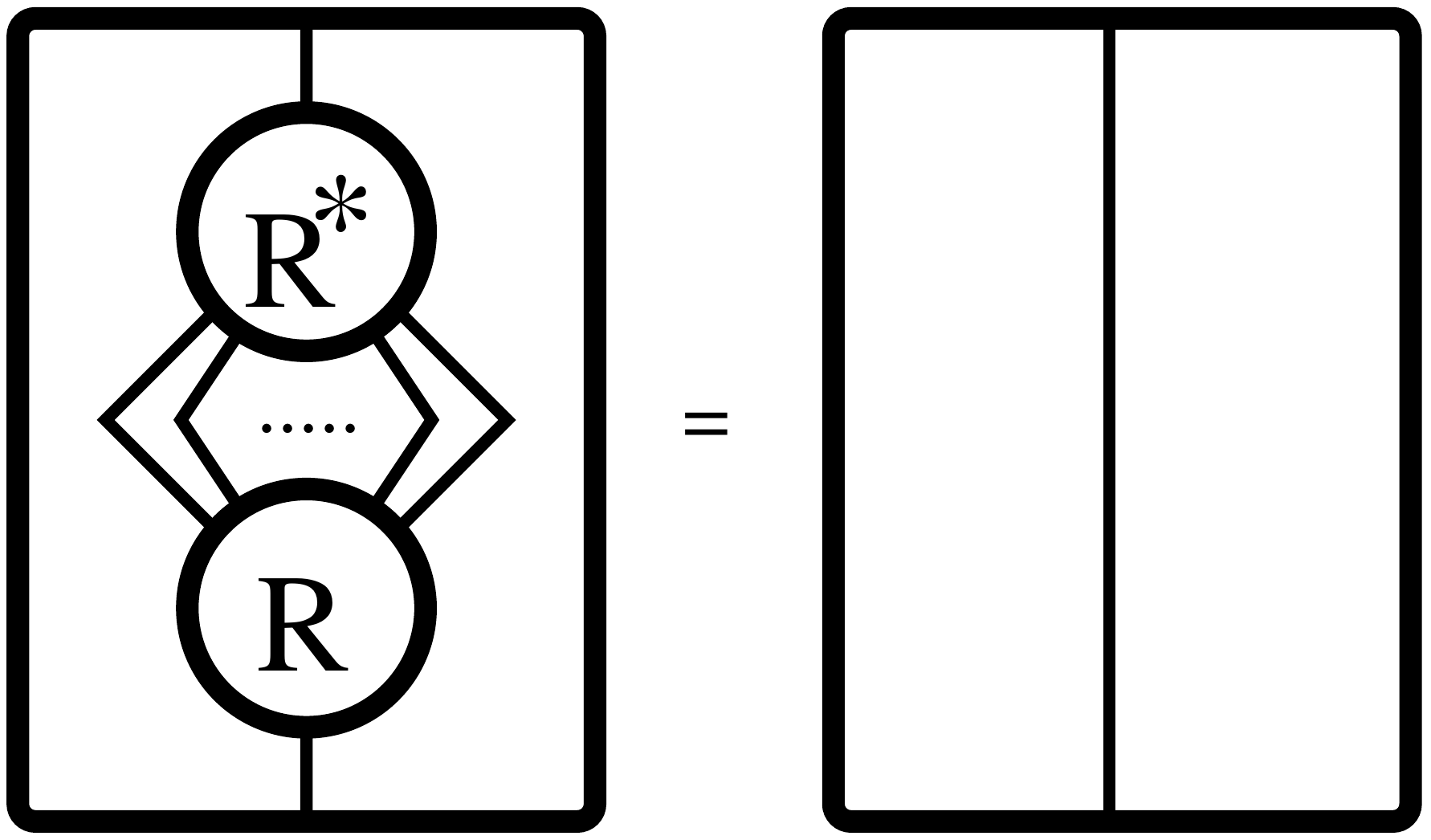} {2in}
 
 \end{myequation}
 Once we have fixed $R$ with the unitarity condition as above we can suppress it, and the input and output discs, from the diagrams so that the above condition becomes simply:

 \vskip 5pt
 
 \hspace {60pt}\qquad\qquad\qquad\qquad \vpic {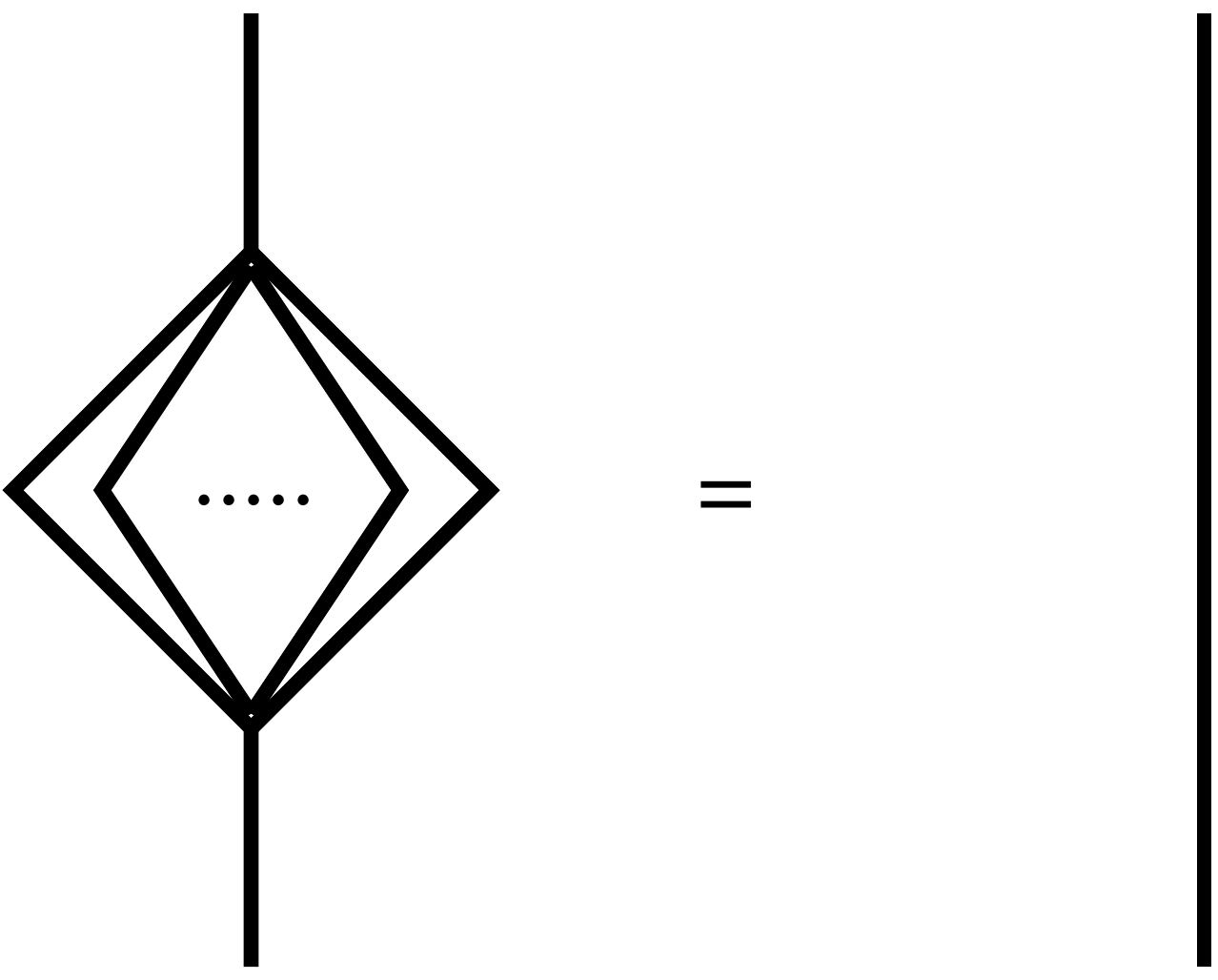} {1in}
 
 There are many variants on how to make forests act on  the $P_n$.
 For instance  we can make the forest $f_{p,q}$ act from $P_{p+1}$ to $P_{q+1}$ by ignoring one boundary point and obtain
 another functor $\Pi_R$ as follows:
 
  \vskip 5pt
\hspace {40pt} \qquad Forest $f:3\rightarrow 6$ \qquad \qquad\quad $\Pi_R(f):P_4\rightarrow P_7$
 \vskip 5pt
\hspace {40pt} \qquad\qquad \vpic {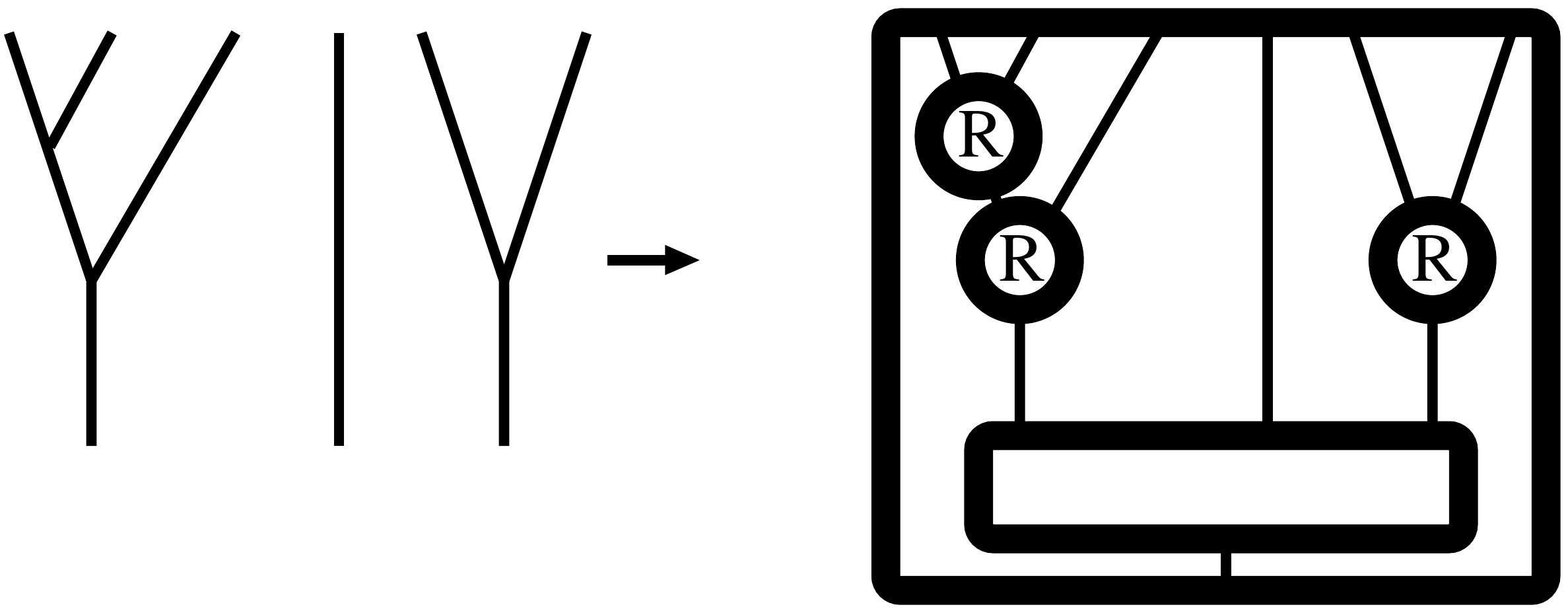} {3in}
 \vskip 5pt

 Both $\Phi_R$ and $\Pi_R$ give unitary representations $\phi_R$ and $\pi_R$ of $F_n$ respectively.
 
 The functor $\Pi_R$ has the advantage that there is a canonical unit vector $\Omega \in P_2$ given by
  $(\mbox{loop parameter}) ^{-1/ 2}$ times a string
 connecting the two boundary points of its disc. It is a simple exercise in the definitions to show then that the \emph
 {coefficient} $\langle \pi_R(g)\Omega, \Omega\rangle$ is given by tying the bottom of the pair of trees defining $g\in F_n$ to
 the top, blowing up the vertices and inserting $R$ and $R^*$ appropriately and evaluating the corresponding element of
 $P_0$. For instance if $g$ is the $X$ in the introduction we get:
 \begin{myfigure}\label{basicformula}
 $$\langle \pi_R(g)\Omega, \Omega\rangle={1\over \mbox{loop parameter}} \vpic {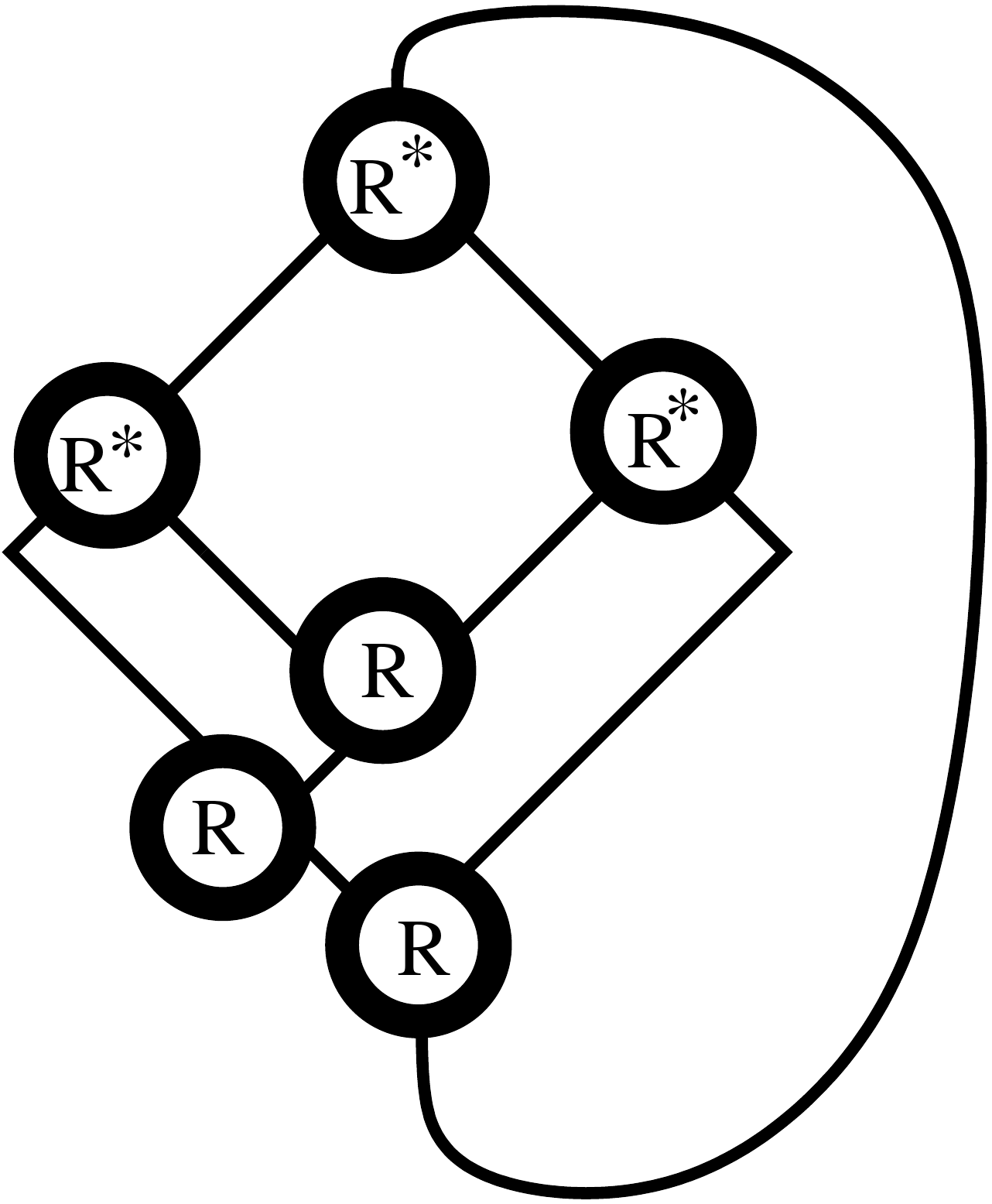} {1.2in} $$
 \end{myfigure}
 It is these representations $\pi_R$ that we will work with. The following theorem contains the only outcome of 
 the above construction that we will use.
 \begin{theorem}\label{reps} Given a Hilbert planar algebra $(P_n)$  and an element $R\in P_{n+1}$ satisfying
 equation \ref{unitarity} there is a unitary representation $\pi_R$ of the Brown-Thompson group $F_n$ on a Hilbert
 space ${\mathcal H }_R$ together with a vector $\Omega \in {\mathcal H }_R$ so that, if $\displaystyle g={ s\over t}\in F_n$ is given
 by a pair of trees then the coefficient $\langle\pi_R(g)\Omega, \Omega\rangle$ is equal to the evaluation in $P_0$ of
 the pair of trees interpreted in $P_0$ via insertions of $R$ and $R^*$ as an element of $P_0$ with the roots of $s$ and $t$ joined to each other as in figure 
 \ref{basicformula}.

 \end{theorem}
Once $R$ is fixed, the subscripts of ${\mathcal H}_R$ etc. will be suppressed.

 A word of warning: when calculating the moments $\langle \pi(g^k)\Omega, \Omega\rangle$ we will use Guba-Sapir-Belk composition method, however \emph{one cannot replace the vertices of the diagrams until all the cancellations have been made since 
these cancellations do not hold in the planar algebra}.

An interesting  example is the following, if $R$ is the (symmetric, self-adjoint) tensor with three indices taking values $1,2$ and $3$ :

 $$R_{i,j,k} =\begin{cases} \frac{1}{\sqrt 2} & \mbox{  if  } i,j,k \mbox{ all distinct  }\\
                                            0   &   \mbox{ otherwise }
                       \end{cases}  $$
                      
then  $R$ satisfies the unitarity condition and  $\langle \pi(g)\Omega, \Omega\rangle$ is the number of 3 edge-colourings of the planar graph given by the pair of
trees with the top connected to the bottom.  

 The non-vanishing of $\langle \pi(g)\Omega, \Omega\rangle$ for all $g\in F$ is
known to be equivalent to the four colour theorem!

 \section{Some spectral measures.}\label{section-examples}
Although our general results on spectral measure hold in all cases we will be especially interested in this paper
in the case where the planar algebra is that of a trivalent graph as in \cite{MPS}. Unfortunately this paper does not consider
the unitary structure but  the quantum $SO(3)$  category   can easily be obtained from the Temperley-Lieb planar algebra
whose unitarity goes back to  \cite{jo1}. Thus this planar algebra $(P_n)$ is
 a Hilbert planar algebra provided the loop parameter $d$ (to distinguish
from the $\delta$ of the Temperley-Lieb) is in the set $\{(4\cos^2 \pi/k)-1\; |\; k=6,7,8,9,\cdots\}\cup [3,\infty)$.

Thus we will fix the trivalent vertex in this category as the element $R\in P_3$ which is a
self-adjoint and rotationally invariant element  of $P_3$ and consider the representation $\pi$ of $F$ given by Theorem \ref{reps}
with its privileged unit vector $\Omega$.

In general if $u$ is a unitary operating on a Hilbert space with a vector $\psi$, the spectral theorem says there
is a Radon measure $\mu_\psi$ on the circle $\mathbb T$ such that, for any continuous function $f:\mathbb T\rightarrow \mathbb C$, $$
\langle f(u) \psi, \psi\rangle= \int_{\mathbb T} f(\theta) d\mu_\psi(\theta).$$
 In this section we determine the spectral measures for three elements of the Thompson group $F=F_2$ given
 by $\Omega$.  
  
 We recall some relations from \cite{MPS}: 
  \begin{formula} \label{mps}
  \mbox{   }\\
\noindent$     $\vpic {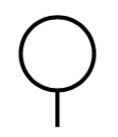} {.2in} $ = 0$, \vpic {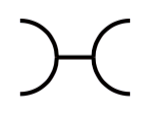} {.2in} $ - $ \vpic {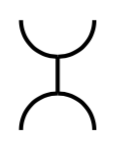} {.2in} $ = \frac{1}{d-1} {\bigg(} $ \vpic {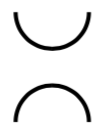} {.2in} $ - $  \vpic {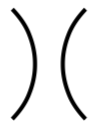} {.2in} $\bigg), $   \vpic {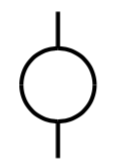} {.2in} $ = $ \vpic {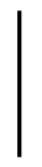} {.07in} \mbox{ and } \vpic {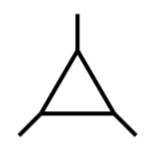} {.28in} $ = t $  \vpic {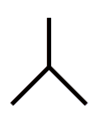} {.2in} ,
\5
\noindent where $t=(d-2)(d-1)^{-1}$, $d=\delta^2-1$ and $\delta = 2\cos\pi/n$  for $n=6, 7, 8, 9, \ldots$
\end{formula} 
In  our examples, we will  find a suitable representation of   the element in $F$, calculate the moments, and finally reconstruct the measure.  
\begin{example} We let $A$ be the usual generator of $F$ given by the pair of trees picture \\
$\mbox{            }  \qquad\qquad\qquad\quad\qquad\qquad\qquad  A=$\vpic {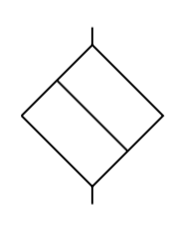} {0.75in} .\\
Using the Guba-Sapir-Belk multiplication or otherwise it is easy to see that  $A^n$ is given by the pair of trees picture\\ 
$\mbox{            } \qquad\quad\qquad \qquad\qquad\qquad    A^n=$\vpic {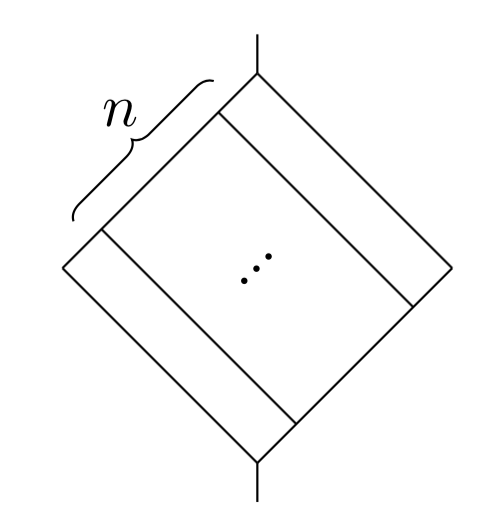} {1.25in} .\\
So by the last formula of \ref{mps} we see that $$ \langle \pi(A)^n \Omega, \Omega\rangle=t^{|n|}.$$

Now let  $\mu$ be a measure on $\mathbb T$ with moments $\mu_n$, i.e. $\int_{\mathbb T} e^{in\theta} d\mu=\mu_n$.
Then if $\sum_{n=-\infty}^\infty \mu_n e^{in\theta}$ converges to the smooth function $f$ then $\int_{\mathbb T} f(\theta)e^{-in\theta} d\mu=\mu_{-n}$
and since a Radon measure is determined by its values on the $e^{in\theta}$, $\mu$ is absolutely continuous w.r.t. Lebesgue 
measure and  $\int_{\mathbb T} e^{in\theta} d\mu=  \int_{0}^{2\pi}\bar f(\theta) e^{in\theta} d\theta/2\pi$. 

Thus we see that the spectral measure of $\pi(A)$ is absolutely continuous and its density is given by the Poisson kernel
$$\sum_{n\in \mathbb Z} t^{|n|}e^{in\theta}=\frac{1-t^2}{1-2t\cos\theta +t^2}.$$
\end{example}
 \begin{example}\label{ese-1}
 Consider the following element of $F$\\
 $\mbox{                       } \qquad \qquad \qquad \qquad \; \quad  N =  $\vpic {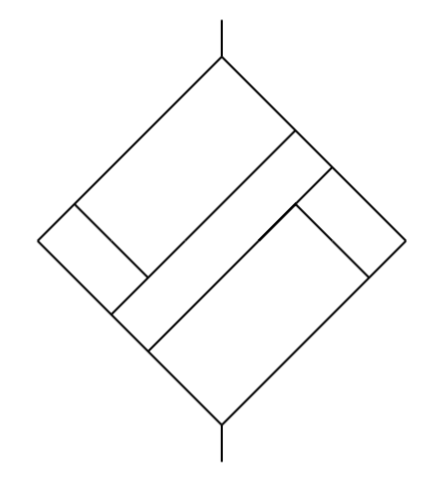} {1.75in}
 \\
By applying an inverse move of type II, we may express $N$ as\\
 $\mbox{                       } \qquad \qquad \qquad \qquad \qquad \quad N = $\vpic {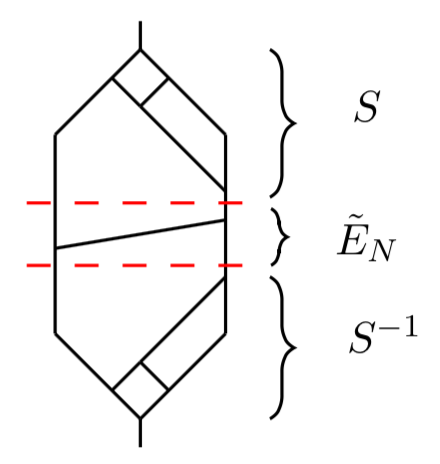} {2in} 
\\
In view of Theorem \ref{BM-description-F}, this representation of $N$ is particularly convenient in calculating its powers because the consecutive occurrences of $S$ and $S^{-1}$ cancel out 
\\
$\mbox{                    }$\qquad $N^k =$ \vpic {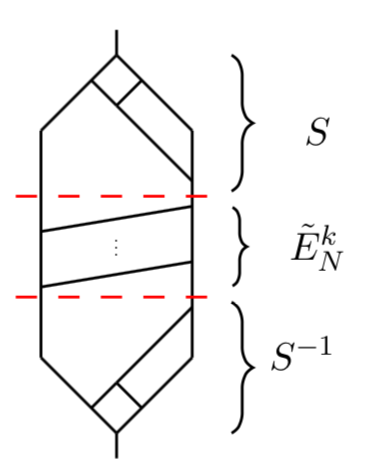} {1.9in} $=$ \vpic {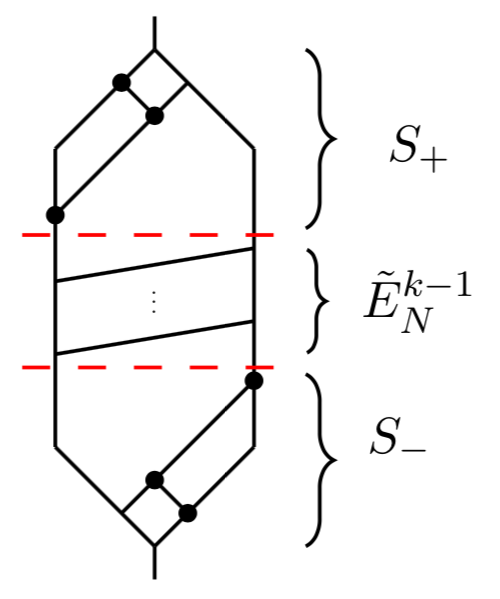} {1.9in}  $\forall k\geq 1$
\\
We observe that the last diagram is reduced.\\
We are now in a position to find the moments of $N$ in our representations.
For $k\geq 1$ we have \\
$\mbox{                       }    \qquad \langle\pi(N)^k\Omega,\Omega\rangle  \; \Omega  = t^2$ \vpic {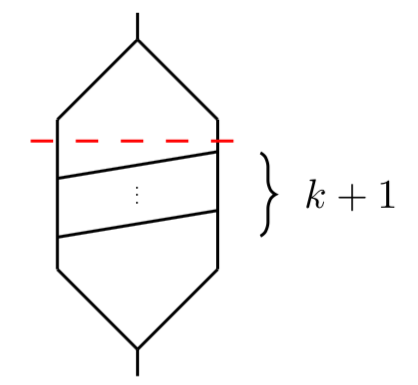} {2in} $=t^2\langle \pi(A)^{-k-1}\Omega,\Omega\rangle  \; \Omega$ ,
\\
 We already know  the moments of the generator $A$ so in some sense we are done. \\
But we would like to reinterpret the calculation of the moments of $A$ in a way that will work in more generality.

If we crop the  above diagram under the red dashed line we get the following elements in the planar algebra
\\
$\mbox{                       }     \qquad\qquad\qquad\qquad v= $\vpic {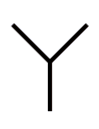} {.2in} , \quad $\tilde E_N^{k+1} v=$  \vpic {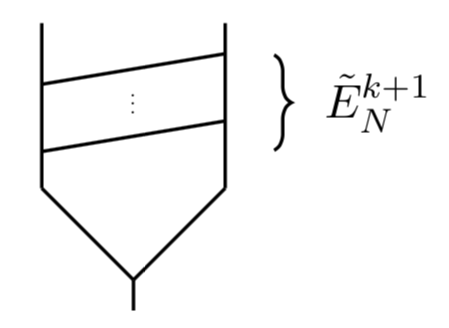} {2in} .
\\
In the planar algebra, we may see  $\tilde E_N= $\vpic {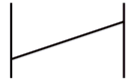} {.25in} acting as 
 an endomorphism
 of the $1$-dimensional vector space spanned by $v$,
the action being  $\tilde E_A v= tv$.
Thus we can interpret $\langle \pi(A)^{-k-1}\Omega,\Omega\rangle$ as the inner product 
$$
\langle \tilde E_N^{k+1}v,v\rangle = \langle t^{k+1}v,v\rangle = t^{k+1}\langle v,v\rangle = t^{k+1}
$$
Therefore, we have $\langle\pi(N)^k\Omega,\Omega\rangle = t^{|k|+3}$ for all $k\in\mathbb{Z}\setminus \{0\}$ and $\langle \pi(N)^0\Omega,\Omega\rangle = 1$. 
Then, the spectral measure of $g$ is $d\mu=fd\theta/(2\pi)$, where 
$f$ is the function
\begin{align*}
f(\theta) & = \sum_{n=-\infty}^\infty \mu_n e^{in\theta} =\frac{2t^3-2t^4\cos\theta}{1-2t\cos\theta+t^2}-2t^3+1
\end{align*}

 \end{example}

 \begin{example}\label{ese-2}
 We now consider the element $X\in F$ shown in 
 \eqref{elementX}, which we may see as the composition of     three strand diagrams\\
   $\mbox{                       }    \; \; \qquad\qquad\qquad\qquad \qquad\qquad \; \; $\vpic {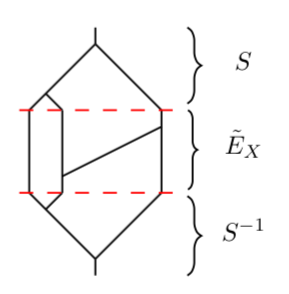} {1.75in}
\\
By using this representations of $X$, its powers assume the following form\\
\begin{eqnarray}	\label{powers-2}
   \qquad\quad\; \qquad\qquad\qquad \qquad\qquad\vpic {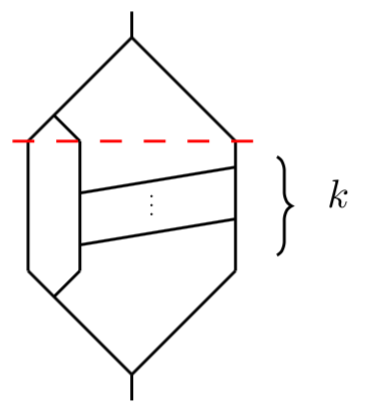} {1.5in} \qquad\qquad\qquad\qquad	k\geq 1\; .
\end{eqnarray}
If we crop the figure under the red dashed line, we get these two elements of the planar algebra:\\
   $\mbox{                       }    \qquad\qquad \qquad\qquad \quad \xi = $ \vpic {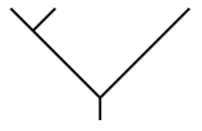} {0.75in}, $\eta=$ \vpic {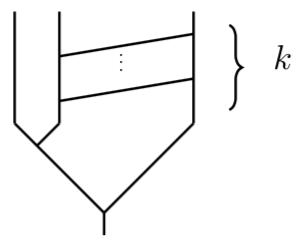} {1in} 
\\
We may interpret  $\tilde E_X:=$\vpic {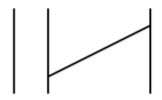} {.4in} as a linear map of 
the three dimensional vector space $V={\rm span}\{v_1,v_2,v_3\}$, where  $v_1 = $\vpic {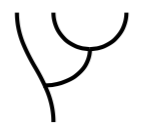} {.2in} , $v_2 = $\vpic {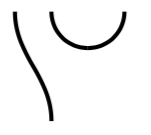} {.2in} , $v_3 = $\vpic {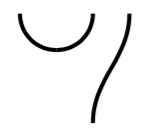} {.2in} . It follows that $\eta=\tilde E_X^k\xi$. 
Let us diagonalise the linear map $\tilde E_X$ and find a formula for all the moments of $X$.
With respect to this basis, the map $\tilde E_X$ is represented by the following matrix\\
$$
\tilde E_X = \left( 
\begin{array}{ccc}
t & 0 & 1\\
0 & 1 & (d-1)^{-1}\\
0 & 0 & (1-d)^{-1}\\
\end{array}
\right)\; .
$$
which has eigenvalues and eigenvectors
\begin{align*}
\lambda_1 & = t  &  w_1 &  = v_1\\
\lambda_2 & = 1 &  w_2 & = v_2\\
\lambda_3 & = (1-d)^{-1}  & w_3 & = -v_1-v_2/d+v_3
\end{align*}
With respect to the basis $\{w_1, w_2, w_3\}$, the linear map $\tilde E_X$ is represented by the diagonal matrix ${\rm diag}(t,1,(1-d)^{-1})$.
Simple computations show that 
$\xi =tw_1+d^{-1}w_2+(1-d)^{-1}w_3$ and thus, 
for $k\geq 1$, we have $\eta=\tilde E_X^{k}\xi
=t^{k+1}w_1+d^{-1}w_2+(1-d)^{-k-1}w_3$.
\\
As in the previous example, we interpret \eqref{powers-2} as 
$\langle 	\pi(X)^{k}\Omega,\Omega\rangle  
=  \langle \tilde E_X^{k}\xi,\xi\rangle$. Thus the moments are\\
$    
\mbox{                                              }  \langle \pi(X)^{k}\Omega,\Omega\rangle  \; \Omega  \; = t^{k+1} $\vpic {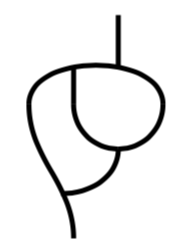} {0.2in} $+\; \frac{1}{d} $ \vpic {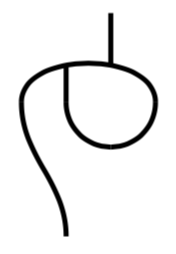} {0.24in} $-\;  \frac{1}{(1-d)^{k+1}}$ \vpic {fig42} {0.24in}  $-\;  \frac{1}{d(1-d)^{k+1}}$ \vpic {fig43} {0.24in} 
 $$
 = \left(t^{k+2}+\; \frac{1}{d} -\;  \frac{t}{(1-d)^{k+1}}-\;  \frac{1}{d(1-d)^{k+1}}\right) \; \Omega
 $$
Now there are two cases to handle depending on whether $d=2$ (which entails that $t=0$) or not. In the former case we have
 $\langle \pi(X)^{k}\Omega,\Omega\rangle =(1+(-1)^k)/2$. This means that the measure is actually $\mu=(\delta_1+\delta_{-1})/d$, where $\delta_{a}$ denotes the Dirac measure with support $\{a\}$.
In the latter case the spectral measure is given by the sum of $\delta_1/2$ and $fd\theta/(2\pi)$, where 
\begin{align*}
f(\theta) 
& = \frac{2t^2-2t^3\cos\theta}{1+t^2-2t\cos\theta}-\frac{(t+d^{-1})(2-2d-2\cos\theta)}{(1-d)^2-2(1-d)\cos\theta+1} 
+\frac{2td^2-2td+1-d^2+2d}{d(1-d)}
\end{align*}
\end{example}
The aim of the next sections is to generalize the method used in these examples to any element of the Thompson group and to a broader class of representations.

 \section{The essential part of an element of $F_n$.}
 Motivated by the calculations of the previous section we will give a more or less canonical decomposition
 of an element $g\in F_n$ which will allow us to control the ``width'' of $g^k$ for large $k$.
 
 As in \cite[Section 3]{BM} we will work with the groupoid of equivalence classes of   $(h,k)$-strand diagrams. 
Given an $(h,k)$-strand diagram $f$ we denote by $[f]$ the corresponding element of the groupoid.

\begin{theorem}\label{lemma-essential}
Every $g$ in $F_n$ can be written in the form
  \\
   $\mbox{                       }    \; \; \quad\qquad\qquad\qquad \qquad\qquad \quad g=$\vpic {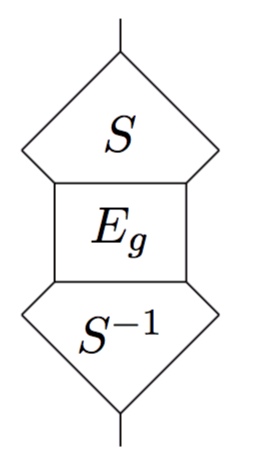} {1.25in}
\\
where $S$ is a $(1, m)$-strand diagram, $E_g$ is a reduced $(m,m)$-strand diagram such that 
if we concatenate two copies of $E_g$ the resulting diagram is already reduced.
\end{theorem}
The Example \ref{ese-1} shows that, in general, $S$ does  not need to be a tree, but just a strand diagram.\\
As in  \cite[Section 4]{BM} one can define three types of cycles in annular strand diagrams, that is  free loops (directed cycle with no vertices),  split loops (a directed cycle with splits, but no merges), and merge loops (a directed cycle with merges, but no splits).
Before proving the previous theorem, we state a couple of results whose proofs are straightforward extensions of those in \cite{BM}. 
\begin{proposition}\label{prop-red-conj}
{\rm (cf. \cite[Proposition 3.2, p. 249]{BM})}
Let $f$ be a $(p, p)$-diagram, let $f'$ be its closure,
and let $g'$ be a closed diagram obtained by applying a reduction to $f'$. Then there exist a $(q, q)$-diagram $g$ whose closure is $g'$, and an $(p,q)$-strand diagram $h$ such that $[f]=[h][g][h]^{-1}$.
\end{proposition}
\begin{proposition}\label{prop-structure}
{\rm (cf. \cite[Proposition 4.1, p. 252]{BM})}
 Let $f$ be any reduced closed strand diagram. Then
\begin{enumerate}
\item Every component of $f$ has at least one directed cycle.
\item Every directed cycle in $f$ is either a free loop,  
 a split loop,  
 or a merge loop. 
\item Any two directed cycles in $f$ are disjoint, and no directed cycle intersects itself.
\end{enumerate}
\end{proposition}
\begin{proof}[Proof of Theorem \ref{lemma-essential}]
Let $g'$ be the closure of $g\in F_n$.
There exists a finite number of reduction moves of type I, II, and III, say $k$, that yield a reduced annular strand diagram.
By Proposition \ref{prop-red-conj} 
 there is a family of elements of the groupoid $g_0=g, g_1, \ldots , g_k$ and $h_1, \ldots , h_k$ such that, for all  $i=0, \ldots , k-1$, we have $[g_i]=[h_{i+1}][g_{i+1}][h_{i+1}]^{-1}$, 
the closure of $g_{i+1}$ is obtained from the closure of $g_i$ by applying a reduction move,
the closure of $g_k$ is the reduced annular strand diagram corresponding to $g'$.
\\
Therefore, we have
$$
[g]=[h_1][g_1][h_1]^{-1}=[h_1][h_2][g_2][h_2]^{-1}[h_1]^{-1}=\ldots =[h_1]\cdots [h_k][g_k][h_k]^{-1}\cdots [h_1]^{-1}
$$
We set $[S]:=[h_1]\cdots [h_k]$ and $[E_g]:=[g_k]$.

That no cancellation can occur taking powers of $E_g$ is clear. Take 
the universal cover $\pi :\tilde A\rightarrow A $ of an annulus $A$ containing $E_g$. Then $\pi^{-1}(E_g)$ consists of an infinitely long chain of copies of $E_g$. If there were any cancellation between vertices in $E_g^p$, $\pi$ would send them to
a cancellation (perhaps wrapping round $A$ several times) in $E_g$ inside $A$. But $E_g$ is supposed to be reduced in the annulus.

\end{proof}
The following lemma is interesting in itself and provides the last tool  we need for the proof of the main result of this section.
\begin{lemma} \label{no-cancellations}
Let $A$ and $B$ be composable reduced strand diagrams. 
There is a sequence of cancellations in $AB$ between vertices in $A$ and vertices in $B$ so that
the result $\overline{AB}$ of performing these cancellations admits no more cancellations.
\end{lemma}
\begin{proof} It suffices to take a sequence $\mathcal S$ of cancellations between vertices in $A$ and $B$ that is maximal.
By contradiction suppose that, after the cancellations in $\mathcal S$ have been performed, there is a cancellation between
vertices $v$ and $w$, which must be both in $A$ or both in $B$ by the maximality of $\mathcal S$.

 First suppose this cancellation is a move of type II. Wolog suppose $v$ is the merge and that both $v$ and $w$ are in $A$. Then in
$AB$ the edge emerging from $v$ must connect $v$ to a vertex $u$ of $A$ other than $w$, otherwise $A$ itself would admit cancellations.
($v$ can never be connected to a vertex of $B$ as we will see.)
But since $u$ does not appear in $\overline{AB}$, there is a cancellation in $\mathcal S$ involving $u$. This cancellation $s$ must be with
a vertex of $B$. So whether $s$ is of type I or II,  after performing $s$ there is an edge between $v$ and a vertex in $B$  (possibly a sink). But no subsequent cancellations
can result in an edge connecting $v$  to a vertex outside $B$. Thus the edge in $AB$ connecting $v$ to $w$ is impossible, a contradiction.

Now suppose the cancellation between $v$ and $w$ is type I and let $v$ be the split. One of the edges emanating from $v$ must have
been originally connected to some other vertex $u$ of $A$. Now argue as before.
\end{proof} 
\begin{theorem}\label{prop-structure-powers}
Let $g,h,\tilde h$ be three elements of $F_n$. Then, there exists a $k\in\mathbb{N}$ and two strand diagrams $S_+$, $S_-$, such that for all $p\geq k$ the following diagram is reduced
  \\
   $\mbox{                       }    \; \; \quad\qquad\qquad\qquad \qquad\qquad \quad h g^p \tilde h=$\vpic {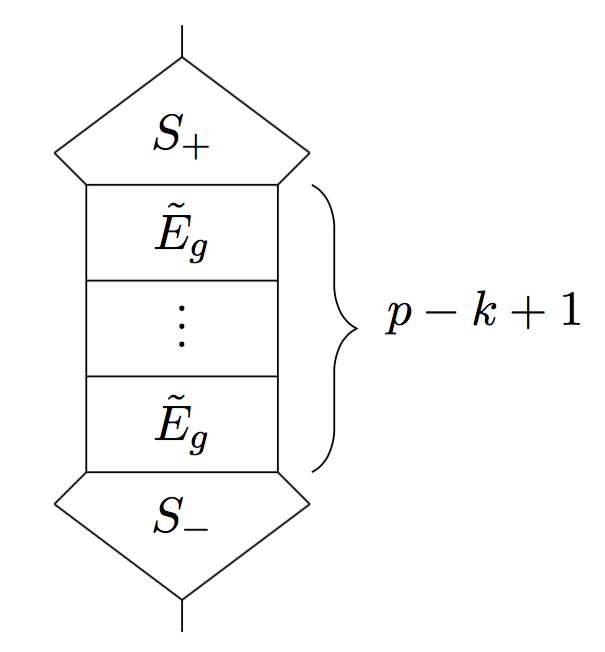} {1.75in}
\\
where $\tilde E_g$ is a reduced $(m,m)$-strand diagram.  
\end{theorem}
\begin{proof}
By using the representation of $g$ in Lemma \ref{lemma-essential}, we see that the consecutive  occurrences of $S^{-1}$ and $S$ cancel out 
 for all $p\in\mathbb{N}$. Therefore, 
 we get
  \\
   $\mbox{                       }    \; \;  \qquad  \quad g^p=$\vpic {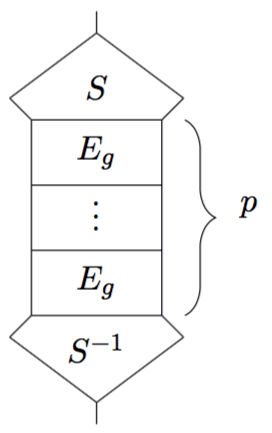} {1.35in} \quad and
  $\mbox{                       } \quad  h g^p \tilde h =  $\vpic {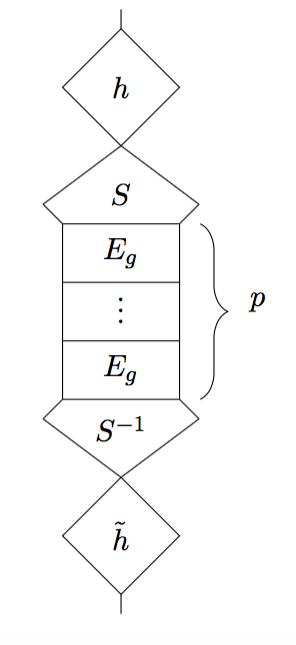} {1.5in} 
\\
Now let $\hat S$ and $\check S$ be the reduced strand diagrams associated with $hS$ and $S^{-1}\tilde h$, respectively.
Consider the pairs $A_p:=\hat S$, $B_p:=(E_g)^p$ and apply Lemma \ref{no-cancellations} to this family of pairs. Since $A_p$ 
is always the same diagram (and has a 
finite
 number of vertices), for $p$ greater than a certain $k_1\in \mathbb N$ there are no more cancellations. Therefore, we have $\hat S (E_g)^p=S_+' (E_g)^{p-k_1+1}$ for all $p\geq k_1$, where  $S_+'$ is a reduced strand diagram.
Similarly, if we consider the pairs $A_p:=(E_g)^p$, $B_p:=\check S$ we  see that, for $p$ greater than a certain $k_2\in \mathbb N$, there are no more cancellations and $(E_g)^p\check S= (E_g)^{p-k_2+1}S_-'$  where  $S_-'$ is a reduced strand diagram.
 This means that for $p\geq k_1+k_2$, we have $h g^p \tilde h =S_+' (E_g)^{p-k_1-k_2+2}S_-'$.
 At this stage 
  some additional reductions might be possible. 
 By construction these reductions can only occur  between vertices of $S_+'$ and $S_-'$.
We observe that 
if  a merge in $S_+'$ 
is connected  by a straight line to a split in $S_-'$, 
a reduction move of type II occurs and  $n-1$ new  parallel edges appear in $E_g$.\\
Similarly 
if a split in $S_+'$ 
 is connected by   $n$ parallel edges to a merge in $S_-'$
 , we may cancel $n-1$ parallel edges of $E_g$ 
thanks to a move of type I. Once these reductions are performed, one easily gets the desired representation of $h g^p \tilde h$.
\end{proof}
We 
 observe that the natural number $k$ and the strand diagrams $S_\pm$, $\tilde E_g$ depend on both $h$ and $\tilde h$. Moreover, $\tilde E_g$ is equal to $E_g$ up to 
 parallel edges. 
When $h=\tilde h=1$,
we 
 call the element $\tilde E_g$ 
  the \emph{essential part of $g$}.
\\

 \section{Absolute continuity of the spectral measure.}
We keep the notation of the previous sections.
Given an element $g\in F_n$, a vector $\psi\in\mathcal{H}$, there is a linear functional $I_{g;\psi}$ which maps  $f\in C(\mathbb T)$ to $\langle f(\pi(g))\psi,\psi\rangle$. As already mentioned, this functional is determined by the so-called spectral measure, which we denote by $\mu_{g;\psi}$, or simply by $\mu_{\psi}$ if the element $g$ is clear from the context.
The aim of this section is to classify the spectral measures which arise from a certain family of representations of $F_n$.  
 \begin{theorem}\label{main-thm}
 With the notations of the previous sections, suppose that $\Phi(k)$ is a finite dimensional Hilbert space for all $k\in\mathbb N$ and let $\Omega$ be a vector in $S_\iota$, that is $\Omega$ is of the form $(\iota, x)$, where $\iota$ is   the tree with one vertex and no edges and $x\in\Phi(1)$. Let $g\in F_n$ and $\psi\in [\pi(F)\Omega]^=$, 
 with $\{g_i\}_{i\in I}\subset F_n$, $\alpha_i\in\mathbb C$. 
 Then, 
 the spectral measure associated with $g$ and $\psi$
 can be decomposed as
 $\mu_{g;\psi}=\mu_1+\mu_2$,
 where $\mu_1$ is a measure which is absolutely continuous with respect to the Lebesgue
measure  and $\mu_2$ is a  pure point measure with finite support.
 \end{theorem}
First of all we are going to prove the theorem 
under the assumption that $\psi=\pi(h)\Omega$.
For the sake of clarity, our proof is in turn divided into a series of preliminary lemmas.
Our plan is to 
 find a formula for the moments of $g$, which 
 will allow us to reconstruct the measure thanks to the theory of distributions. 
 \\
 Let us   set some notations. As in Section \ref{section-examples}, we denote the moments by 
 $$
 \mu_p:=\langle \pi(g^p) \psi, \psi\rangle = \langle \pi(g^p h) \Omega, \pi(h)\Omega\rangle =  \langle \pi(h^{-1} g^p h) \Omega, \Omega\rangle
 $$ 
Thanks to Theorem \ref{prop-structure-powers} the element $h^{-1} g^p h$ can be expressed in terms of some diagrams $\tilde E_g$, which depend on  $g$ and $h$. 
By the Jordan-Chevalley decomposition, the operator $\Phi(\tilde E_g)$ can be decomposed
as $\Phi(\tilde E_g)=x_{ss}+x_n$, where $x_{ss}$ is the semisimple part, $x_n$ is the nilpotent part, $x_n^r=0$ for some $r\in\mathbb N\cup\{0\}$, and $[x_n,x_{ss}]=0$. 
We denote by $\{v_l\}_{l\in A}$ an orthonormal basis of eigenvectors for $x_{ss}$ and by $\{\lambda_l\}_{l\in A}$ the corresponding eigenvalues, that is $x_{ss} v_l=\lambda_lv_l$. We observe that $|\lambda_l|\leq 1$ for all $l\in A$ because $\Phi(\tilde E_g)$ may be expressed as the composition of isometries and co-isometries.
\\
By using  Theorem \ref{prop-structure-powers}
the elements $h^{-1} g^p h$ have the following 
 form
\begin{eqnarray}\label{figure3}
\qquad\qquad\qquad h^{-1} g^p h = \vpic {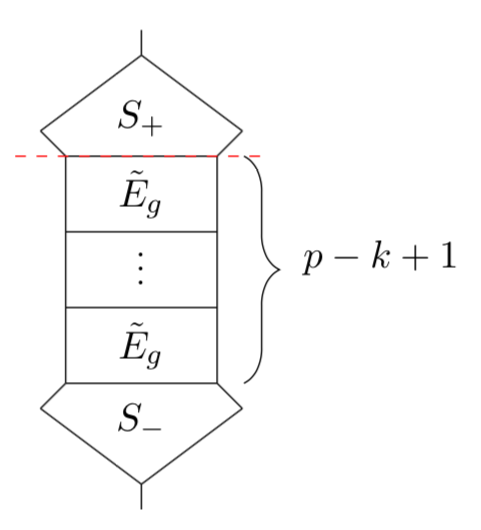} {1.75in} \quad \forall\;  p\geq k
\end{eqnarray}
where $S_{\pm}$ are some strand diagrams and $k$ is a suitable natural number. \\
As done in the examples, we want to use 
the former diagram to compute the moments of $g$.  
 If we crop the diagram \eqref{figure3} in correspondence of the red dashed line, we may express $\mu_p$ as   $\langle \Phi(\tilde E_g)^{p-k+1} \xi, \eta\rangle$, where $\xi=\Phi(S_{-})x$ 
and $\eta =\Phi(S_{+})^*x$. 
We have $\xi =\sum_{l\in A} \xi_l v_l$ and $\eta =\sum_{l\in A} \eta_l v_l$.
\\
The following lemma has to do with the form of the moments.
 \begin{lemma}\label{lemma-moments}
For $p\geq h$ we have that
\begin{equation}\label{formula-moments}
\begin{aligned}
\mu_p & =\sum_{q=0}^{r-1}\sum_{l\in A} \binom{p-h+1}{q}  \xi_l\lambda_l^{p-h+1-q}\langle x_{n}^q v_l,\eta\rangle \\
				& = \sum_{q=0}^{\min\{r-1,p-h+1\}} \sum_{l\in A}  c_{l,q}  \lambda_l^{p-h+1-q} \prod_{k=h-1}^{h+q-2} (p-k)
\end{aligned}
\end{equation}
where $c_{l,q}:=   \xi_l\langle x_{n}^q v_l,\eta\rangle /q!$ and
$\binom{p}{q}=0$ if $q>p$.
\end{lemma}
\begin{proof}
By using    the Jordan-Chevalley decomposition of $\Phi(\tilde E_g)$ we have
\begin{align*}
\mu_p  & =\langle (x_{ss}+x_{n})^{p-h+1} \xi,\eta\rangle=\sum_{q=0}^{p-h+1}\binom{p-h+1}{q} \langle x_{n}^qx_{ss}^{p-h+1-q} \xi,\eta\rangle \\
& =\sum_{q=0}^{r-1}\sum_{l\in A} \binom{p-h+1}{q}  \xi_l\lambda_l^{p-h+1-q}\langle x_{n}^q v_l,\eta\rangle\; .
\end{align*}
\end{proof}
The 
 linear functional $I_{g;\psi}: C(\mathbb T)\to \mathbb C$ 
 is determined by its restriction to  the  Laurent polynomials in $e^{i\theta}$, here denoted by $\mathbb C[e^{i\theta},e^{-i\theta}]$. In particular, we have
$I_{g;\psi}(e^{ik\theta}):=\mu_k$ and $I_{g;\psi}(e^{-ik\theta}):=\mu_{-k}=\bar\mu_k$ for $k\geq 0$. 

We may partition $A$ into two subsets
$$
A_1 =\{l\in A \; | \; |\lambda_l|<1\}\qquad A_2 =\{l\in A \; | \; |\lambda_l|=1\}\; .
$$
and define three linear functionals on $C(\mathbb{T})$ by
\begin{align*}
I_{g,1;\psi}(e^{ip\theta}) & :=\left\{ \begin{array}{ll} 
 		\sum_{q=0}^{\min\{r-1,p-h+1\}}\sum_{l\in A_1} 
c_{l,q} 
\lambda_l^{p-h+1-q} \prod_{k=h-1}^{h+q-2} (p-k) &  \textrm{ for $p\geq h$}\\
		0 & \textrm{ for $0\leq p<h$}
 		\end{array}\right.
		\\
I_{g,2;\psi}(e^{ip\theta}) & :=\left\{ \begin{array}{ll} 
 		\sum_{q=0}^{\min\{r-1,p-h+1\}}\sum_{l\in A_2} 
  c_{l,q}
 \lambda_l^{p-h+1-q} \prod_{k=h-1}^{h+q-2} (p-k) & \textrm{ for $p\geq h$}\\
		0 &  \textrm{ for $0\leq p<h$}\\
 		\end{array}\right.\\
I_{g,3;\psi}(e^{ip\theta}) & :=\left\{ \begin{array}{ll} 0 &  \textrm{ for $p\geq h$}\\
	\mu_p &  \textrm{ for $0\leq p<h$}\\
 		\end{array}\right.
\end{align*}
Accordingly, 
we have the decomposition $I_{g;\psi} = I_{g,1;\psi}+I_{g,2;\psi}+I_{g,3;\psi}$,
\bigskip

\begin{lemma}\label{lemma-case-1}
The functional $I_{g,1;\psi}: C(\mathbb T)\to\mathbb C$ is of the form $I_{g,1;v}(f)=\int_{0}^{2\pi}\; f(\theta) \; \bar f_1(\theta)d\theta/(2\pi)$ for some $f_1\in C(\mathbb T)$.
\end{lemma}
\begin{proof}
By formula \eqref{formula-moments} and the hypothesis $|\lambda_i|<1$ for all $i\in A_1$, we have that the following series is absolutely convergent
\begin{align*}
& \sum_{p\in \mathbb{Z}} I_{g,1;\psi}(e^{ip\theta})e^{ip\theta}  = \sum_{p\geq h}\sum_{q=0}^{\min\{r-1,p-h+1\}}\sum_{ l\in A_1}   c_{l,q}   
 \lambda_l^{p-h+1-q} e^{ip\theta}  \prod_{k=h-1}^{h+q-2} (p-k) \\
 &  \quad  +\sum_{p\geq h}\sum_{q=0}^{\min\{r-1,p-h+1\}}\sum_{  l\in A_1}    \bar  
 c_{l,q}  
 \bar \lambda_l^{p-h+1-q} e^{-ip\theta}  \prod_{k=h-1}^{h+q-2} (p-k)
\end{align*}
We denote by $f_1(\theta)$ the corresponding function in $C(\mathbb T)$. 
It 
follows that the functional $I_{g,1;\psi}$ is induced by the measure $\bar f_1d\theta/(2\pi)$.
\end{proof}

\begin{lemma}\label{lemma-case-2}
It holds
$$
I_{g,2;\psi}(\cdot)=\sum_{q=0}^{r-1}  \sum_{l\in A_2}  c_{l,q} (-1)^q\delta_{\lambda_l}^{(q)}(e^{i(-h+1)\theta}\; \cdot \; )-\int_{0}^{2\pi}\; \cdot \; \tilde f(\theta)d\theta/(2\pi)
$$
where $\delta_{\lambda_l}^{(q)}(e^{ip\theta} f):=(-1)^q (e^{ip\theta} f)^{(q)}(\lambda_i)$, $\tilde f(\theta)\in \mathbb C[e^{i\theta},e^{-i\theta}]$. 
\end{lemma}
\begin{proof}
It is enough to check the equality on $e^{ip\theta}$, $p\geq h$. 
Indeed, we have
\begin{align*}
& \sum_{q=0}^{r-1}   \sum_{l\in A_2}  c_{l,q}  (-1)^q\delta_{\lambda_l}^{(q)}(e^{i(p-h+1)\theta}) = \\
& =\sum_{q=0}^{\min\{r-1,p-h+1\}}  \sum_{l\in A_2}  c_{l,q} 
\lambda_l^{p-h+1-q}  \prod_{k=h-1}^{h+q-2} (p-k)
= I_{g,2;\psi}(e^{ip\theta})\; .
\end{align*}
\end{proof}
We are at last in a position to prove the main result of this paper.
\begin{proof}[Proof of Theorem \ref{main-thm}]
By the previous discussion, it follows that 
$I_{g;\psi}$ is a distribution of order strictly greater than $0$, unless
$$
\sum_{q=0}^{r-1}   \sum_{l\in A_2}   c_{l,q} (-1)^q\delta_{\lambda_l}^{(q)}(e^{i(-h+1)\theta}\; \cdot \; )=0
$$
Since a measure is a distribution of order $0$, then we have proven the theorem under the hypothesis that $\psi=\pi(h)\Omega$.

Suppose now that $\psi$ is a generic vector in $[\pi(F)\Omega]^=$. 
The Hilbert space $\mathcal H$   can be decomposed into $\mathcal H_{ac}\oplus\mathcal  H_{sc}\oplus \mathcal H_{pp}$, where $\mathcal H_{ac}=\{v\in \mathcal H\; | \;  \mu_{g;v}$ is absolutely continuous$\}$, $\mathcal H_{sc}=\{v\in \mathcal H\; | \;  \mu_{g;v}$ is  continuous singular$\}$, $\mathcal H_{pp}=\{v\in \mathcal H\; | \; \mu_{g;v}$ is pure point$\}$. So far we have shown that there 
is a family of vectors in $\mathcal H_{ac}\oplus\mathcal  H_{pp}$ that span 
a dense subspace of $\mathcal H$.  
 This means that the spectral measure $\mu_{g;\psi}$    cannot contain a component which is continuous singular.\\
All we have to do now is to prove that the support of 
$\mu_2$
is finite.
Thanks to   von Neumann's ergodic theorem we know that, for any contraction $V$ on an Hilbert space and any  $\lambda\in\mathbb{T}$, the sequence $\Sigma_N(\lambda^{-1} V):=\sum_{i=0}^N (\lambda^{-1} V)^i/N$ converges in the strong operator topology to the projection onto the eigenspace Ker$(V-\lambda)$. 
By Theorem \ref{prop-structure-powers} we may express $h_2^{-1}g^ph_1$ as $S_+\tilde{E}(h_1,h_2)^{p-k+1}S_-$ for $p\geq k$,  
where $S_+$, $\tilde{E}(h_1,h_2)$, $S_-$ are suitable strand diagrams.
Despite  the fact that  $\Phi(\tilde{E}(h_1,h_2))$ depends on $h_1, h_2\in F_n$, 
its spectrum is only a function of $\pi(g)$. 
We claim that $\Sigma_N(\lambda^{-1} \pi(g))$ converges to $0$ for every $\lambda$ not in ${\rm sp}(\Phi(\tilde E_g(h_1,h_2)))\cap \mathbb{T}$. 
To this end it is enough to show that the limit is $0$ in the weak topology and actually
it suffices to show $\langle \Sigma_N(\lambda^{-1} \pi(g)) \pi(h_1)\Omega , \pi(h_2)\Omega\rangle\to 0$ for every $h_1$, $h_2\in F_n$.
Simple computations lead to
\begin{align*}
\langle \Sigma_N(\lambda^{-1} \pi(g)) \pi(h_1)\Omega , \pi(h_2)\Omega\rangle 
& = \langle \sum_{i=0}^N \frac{(\lambda^{-1}\pi(g))^i}{N} \pi(h_1)\Omega , \pi(h_2)\Omega\rangle\\
& = \langle \sum_{i=0}^N \frac{(\lambda^{-1}(\Phi(E_g(h_1,h_2)))^i}{N} \xi_1,\xi_2\rangle\\
& = \langle \Sigma_N(\lambda^{-1}  \Phi(E_g(h_1,h_2))) \xi_1,\xi_2\rangle
\end{align*}
Since $\Phi(E_g(h_1,h_2)))$ is an endomorphism of a finite dimensional vector space, it has finitely many eigenvalues and  $\Sigma_N(\lambda^{-1}  \Phi(E_g(h_1,h_2)))$ has non-trivial limit only for $\lambda\in{\rm sp}(\Phi(\tilde E_g(h_1,h_2)))$. 
From this discussion it follows   that the support of the pure point component is finite. 
\end{proof}

\end{document}